\documentclass[11pt]{amsart}
\usepackage{amsmath,amsthm, amssymb, amsfonts,latexsym,amscd}
\usepackage[dvips]{graphicx}
\newtheorem{theorem}{Theorem}[section]
\newtheorem{corollary}[theorem]{Corollary}
\newtheorem{lemma}[theorem]{Lemma}
\newtheorem{proposition}[theorem]{Proposition}

\newtheorem{remark}[theorem]{Remark}
\newtheorem{example}[theorem]{Example}

\begin{document}
\title[Algebraic connectivity with pendant vertices]{Algebraic connectivity of connected graphs with fixed number of pendant vertices}
\author[A. K. Lal, K. L. Patra and B. K. Sahoo]{}
\address[]{}
\maketitle

\begin{center}
Arbind K. Lal\footnote{
Department of Mathematics and Statistics, Indian Institute of Technology, Kanpur--208016, India; e-mail: arlal@iitk.ac.in}, Kamal L. Patra\footnote{School of Mathematical Sciences, National Institute of Science Education and Research, Sainik School Post, Bhubaneswar--751005, India;
e-mail: klpatra@niser.ac.in} and
Binod K. Sahoo\footnote{School of Mathematical Sciences, National Institute of Science Education and Research, Sainik School Post, Bhubaneswar--751005, India; e-mail: bksahoo@niser.ac.in}
\end{center}

\begin{abstract}
In this paper we consider the following problem: Over the class of all simple connected graphs of order $n$ with $k$ pendant vertices ($n,k$ being fixed), which graph maximizes (respectively, minimizes) the algebraic connectivity? We also discuss the algebraic connectivity of unicyclic graphs.\\

\noindent Keywords: Laplacian matrix; Algebraic connectivity; Characteristic set;
Perron component; Pendant vertex.
\end{abstract}

\section{Introduction}

All graphs considered here are finite, simple and undirected. Let $G$ be a graph with
vertex set $V=\{v_1,\cdots,v_n\}$. The number $n$ of vertices is called the {\it order} of $G$. The {\em adjacency matrix} $A(G)$ of $G$ is defined as $A(G)=[a_{ij}]$, where $a_{ij}$ is equal to one, if the
unordered pair $\{v_i,v_j\}$ is an edge of $G$ and zero, otherwise. Let
$D(G)$ be the diagonal matrix of the vertex degrees of $G$. The {\em
Laplacian matrix} $L(G)$ of $G$ is defined as $L(G)=D(G)-A(G)$. We refer to \cite{mer1,mohar} for a general overview on results related to Laplacians. It is well known that $L(G)$ is a symmetric positive semidefinite $M$-matrix. The smallest eigenvalue of $L(G)$ is zero
with the vector of all ones as its eigenvector. It has multiplicity
one if and only if $G$ is connected. In other words, the second
smallest eigenvalue of $L(G)$ is positive if and only if $G$ is
connected. Viewing the second smallest eigenvalue as an algebraic measure of
connectivity, Fiedler termed this eigenvalue as the {\em algebraic
connectivity} of $G$, denoted $\mu(G)$. The following two lemmas are well known.

\begin{lemma}[\cite{fal3}, p.223]\label{lem:addvertex}
Let $G$ be a graph. Let $\widehat{G}$ be the graph obtained from $G$ by adding a pendant vertex to a vertex of $G.$ Then  $\mu(\widehat{G})\leq\mu(G)$.
\end{lemma}

\begin{lemma}[\cite{fi1}, 3.2, p.299]\label{lem:edge}
Let $G$ be a non-complete graph. Let $G'$ be the
graph obtained from $G$ by joining two non-adjacent vertices of
$G$ with an edge. Then $\mu(G)\leq\mu(G')$.
\end{lemma}

Lemma \ref{lem:edge} implies that, over all connected graphs, the maximum algebraic connectivity occurs for the
complete graph and the minimum algebraic connectivity occurs for a
tree. It is also known that over all trees of order $n$, the path (denoted by $P_n$) has the minimum algebraic connectivity and this minimum is given by $2(1-\cos\frac{\pi}{n})$ (\cite{fi1}, p.304). Merris proved that, among all trees of order $n>2$, the star $K_{1,n-1}$ uniquely attains the maximum
algebraic connectivity which is equal to $1$ (\cite{mer2}, Corollary 2, p.118).

There is a good deal of work on the algebraic connectivity of graphs. We refer to \cite{bapat}--\cite{gr1},\cite{kir1},\cite{patra} for various works on the algebraic connectivity of connected graphs having certain graph theoretic properties. For all graph theoretic terms used in this paper (but not defined), the reader is advised to look at the book Graph Theory by Harary~\cite{harary}.
In this paper, we consider the problem of extremizing the algebraic connectivity over all connected graphs of order $n$ with $k$ pendant vertices for fixed $n$ and $k$, $0\leq k\leq n-1$.

The paper is arranged as follows:
In Section~\ref{preliminaries}, we recall some results from the literature that will be used in the subsequent sections. In Section~\ref{initial}, we record some elementary results and prove a result which gives conditions under which equality
is attained in Lemma \ref{lem:edge}. We study the algebraic connectivity of connected graphs with and without pendant vertices in Section~\ref{with} and Section~\ref{without}, respectively. Finally, in Section \ref{unicyclic}, we discuss the algebraic connectivity of unicyclic graphs.

\section{Preliminaries}\label{preliminaries}

Let $G$ be a connected graph. The {\em distance} $d(u,v)$ between two vertices $u$ and $v$ of $G$ is the length of a shortest path from $u$ to $v$. The {\em diameter} of $G$ is defined by $\max\limits_{u,v}d(u,v)$. A {\em pendant vertex} of $G$ is a vertex of degree one. For a set $W$ of vertices of $G$, $G-W$ denotes the graph obtained from $G$ by removing the vertices in $W$ and all edges incident with them. If $W=\{v\}$ consists of one vertex only, we simply denote $G-W$ by $G-v$ and refer to the connected components of $G-v$ as the {\em connected components of $G$ at $v$}. A vertex $v$ of $G$ is called a {\em cut-vertex} if there are at least two connected components of $G$ at $v$.

An eigenvector of $L(G)$ corresponding to $\mu(G)$ is called a {\em Fiedler vector} of $G$. Let $Y$ be a Fiedler vector of $G$. By $Y(v),$ we mean the
co-ordinate of $Y$ corresponding to the vertex $v.$ A vertex $v$
of $G$ is called a {\em characteristic vertex} with respect to $Y$ if $Y(v)=0$
and there exists a vertex $w$ adjacent to $v$ such that $Y(w)\neq
0.$ An edge $\{u, w\}$ of $G$ is called a {\em characteristic edge} with respect to $Y$
if $Y(u)Y(w)<0$. The {\em characteristic set} of $G$ with respect to $Y$, denoted ${\mathcal C}(G,Y)$, is the
set of all characteristic vertices and edges of $G$. By a result of
Fiedler \cite{fi3}, either ${\mathcal C}(G,Y)$ consists of one vertex only or ${\mathcal C}(G,Y)$ is
contained in a block of $G$. For a tree, it is known that
${\mathcal C}(G,Y)$ is independent of $Y$ and it contains either a vertex or an edge. We refer to \cite{bapat} for more general
results on the size of ${\mathcal C}(G,Y)$.

\begin{figure}[h]
\centering
\includegraphics[scale=.9]{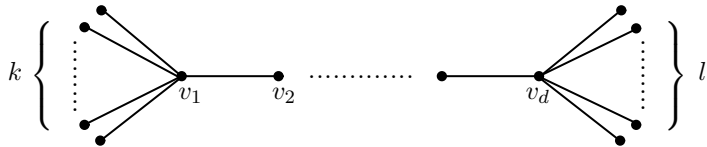}
\caption{The tree $T(k,l,d)$}\label{The tree T(k,l,d)}
\end{figure}

For a fixed positive integer $n,$ the path $ [v_1 v_2\cdots v_n]$ on $n$ vertices is denoted by $P_n$.
For positive integers $k,l,d$ with $n=k+l+d$, let $T(k,l,d)$ be the tree of order $n$ obtained by taking the
 path  $P_d$  and adding $k$ pendant vertices
adjacent to $v_1$ and $l$ pendant vertices adjacent to $v_d$
(see Figure \ref{The tree T(k,l,d)}). The path Note that $T(1,1,d]$ is a path on $d+2$ vertices.
The next proposition determines
the tree, up to isomorphism, which minimizes the algebraic connectivity
over all trees of order $n$ with fixed diameter.

\begin{proposition}[\cite{fal}, Theorem 3.2, p.58]\label{prop:diameter}
Among all trees of order $n$ with fixed diameter
$d+1$, the minimum algebraic connectivity is uniquely attained by $T\left(\lceil\frac{n-d}{2}\rceil,\lfloor\frac{n-d}{2}\rfloor,d\right)$.
\end{proposition}

Now, consider a path $v_1 v_2\cdots v_{d}$ on $d\geq 3$ vertices and add $n-d$ pendant vertices adjacent to the vertex $v_j$, where $j=\lfloor\frac{d+1}{2}\rfloor$. The new graph thus constructed is denoted by $T_{n-d}^{d}$ (see Figure \ref{The tree T_{n-d}{d}}). Then $T_{n-d}^{d}$ is a tree of order $n$ with diameter $d-1$. The next proposition determines the tree, up to isomorphism, which maximizes the algebraic connectivity
over all trees of order $n$ with fixed diameter.
\begin{figure}[h]
\centering
\includegraphics[scale=.9]{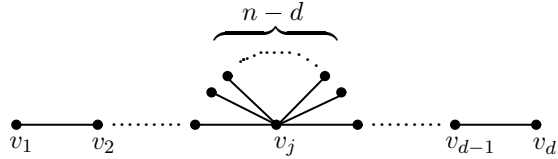}
\caption{The tree $T_{n-d}^{d}$}\label{The tree T_{n-d}{d}}
\end{figure}
\begin{proposition}[\cite{fal}, p.62--65]\label{prop:diameter1}
Among all trees of order $n$ with fixed diameter $d+1$, the maximum algebraic connectivity is uniquely attained by $T_{n-d-2}^{d+2}$.
\end{proposition}

Fix a vertex $v$ of $G$ and let $l,k \geq 1$. We construct a new graph $G_{k,l}$ from $G$ by attaching two new paths $P:vv_{1}v_{2}\ldots v_{k}$ and $Q:vu_{1}u_{2}\ldots u_{l}$ at $v$ of
lengths $k$ and $l$, respectively, where
$u_{1},u_{2},\ldots,u_{l},v_{1},v_{2},\ldots,v_{k}$ are
distinct new vertices. Let $\widetilde{G}_{k,l}$ be the graph obtained
from $G_{k,l}$ by removing the edge $\{v_{k-1},v_{k}\}$ and adding
a new edge $\{u_{l},v_{k}\}$ (see Figure \ref{Grafting an edge}). We say that
$\widetilde{G}_{k,l}$ is obtained from $G_{k,l}$ by {\em grafting} an
edge. The next result compares the algebraic connectivity of $G_{k,l}$ with that of $\widetilde{G}_{k,l}\simeq G_{k-1,l+1}$ whenever the initial graph $G$ is a tree.

\begin{figure}[h]
\centering
\includegraphics[scale=.9]{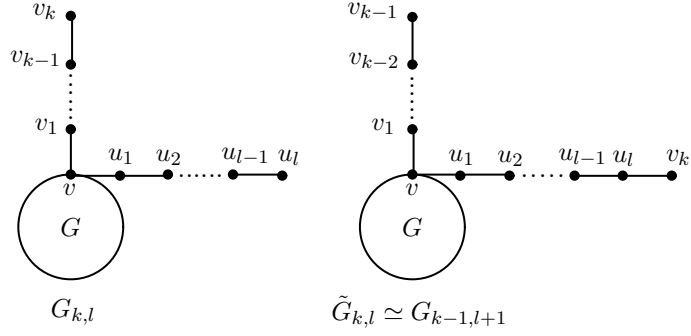}
\caption{Grafting an edge}\label{Grafting an edge}
\end{figure}

\begin{proposition}[\cite{patra}, Theorem 2.4, p.861]\label{prop:gkl}
Let $G$ be a tree of order $n\geq 2$ and $v$
be a vertex of $G$. For $l,k\geq 1$, let $G_{k,l}$ and $\widetilde{G}_{k,l}\simeq G_{k-1,l+1}$ be the graphs as defined above with respect to $v$. If $l\geq k$, then $\mu(G_{k-1,l+1})\leq\mu(G_{k,l})$.
\end{proposition}

Let $v$ be a vertex of $G$ and $C_1, C_2,
\cdots, C_k$ be the connected components of $G$ at $v$. Note that $k\geq2$ if and only if $v$ is a cut-vertex of $G$. For each $i=1,\cdots,k$, we denote by ${\widehat L}(C_i)$ the principal sub-matrix of $L(G)$ corresponding to the
vertices of $C_i$. Since $L(G)$ is an $M$-matrix and has nullity $1$, it follows that ${\widehat L}(C_i)$ is a non-singular $M$-matrix and hence ${\widehat L}(C_i)^{-1}$ is a positive matrix, called the {\em bottleneck matrix} of $C_i$. By the Perron-Frobenius Theorem \cite{minc}, ${\widehat L}(C_i)^{-1}$ has a simple positive dominant eigenvalue,
called the {\em Perron value of $C_i$} and is denoted by $\rho({\widehat L}(C_i)^{-1})$. A corresponding eigenvector
of $\rho({\widehat L}(C_i)^{-1})$ with all entries positive is called a {\em Perron vector of $C_i$}. We say that $C_i$ is a {\em Perron component of $G$ at $v$} if its Perron value is maximum among that the Perron values of $C_1, \cdots, C_k$.

The next proposition gives the description of the entries of the bottleneck matrices of a tree.

\begin{proposition}[\cite{kir2}, Proposition 1, p.313]\label{prop:bottleneck}
Let $T$ be a tree, $v$ be a vertex of $T$ and $C$ be a connected component of $T$ at $v$. Then $\widehat{L}(C)^{-1}=[m_{ij}]$,
where $m_{ij}$ is the number of common edges between the paths $P_{iv}$ joining $i$ and $v$ and $P_{jv}$ joining $j$ and $v.$
\end{proposition}

The following proposition connecting Perron components, bottleneck matrices and
algebraic connectivity recasts some of the results obtained in \cite{fal,kir1}.
Throughout the paper, we denote by $I$ the identity
matrix and by $J$ the matrix of all ones of appropriate orders. For a symmetric
matrix $M$, $\lambda(M)$ denotes its largest eigenvalue.

\begin{proposition}\label{prop:alg}
Let $G$ be a connected graph with a cut-vertex $v$. Let $C_1,C_2,\cdots,C_k$ be the connected components of $G$ at $v$
with $C_1$ as a Perron component. Then the following results hold.
\begin{enumerate}
\item[$(i)$] There is a unique
non-negative number $x$ such that
$$\lambda\left({\widehat L}(C_1)^{-1}-xJ\right)=\rho\left({\widehat
L}(C_2)^{-1}\oplus\cdots\oplus{\widehat L}(C_k)^{-1}\oplus
[0]+xJ\right)=\frac{1}{\mu(G)}.$$

\item[$(ii)$] If there exists a non-negative number $x$ such that
$$\lambda\left({\widehat L}(C_1)^{-1}-xJ\right)=\rho\left({\widehat
L}(C_2)^{-1}\oplus\cdots\oplus{\widehat L}(C_k)^{-1}\oplus
[0]+xJ\right)=\frac{1}{\alpha},$$ then $\alpha$ is an eigenvalue of $L(G)$.
\end{enumerate}
\end{proposition}

%\label{prop:reversealg}

The next result is a special case of Proposition \ref{prop:alg}$(i)$.

\begin{proposition}[\cite{kir1}, Theorem 8, p.146]\label{prop:gamma}
Let $G$ be a connected graph and $\{v_i,v_j\}$ be an edge not on any cycle of $G$. Let $C_i$ be the connected component of $G$ at $v_j$ containing $v_i$
and $C_j$ be the connected component of $G$ at $v_i$ containing $v_j$. Then $C_i$
is the Perron component of $G$ at $v_j$ and $C_j$ is the Perron component of $G$
at $v_i$ if and only if there exists a $\gamma \in (0,1)$ such
that
$$\rho\left({\widehat L}(C_i)^{-1}-\gamma J\right)=\rho\left({\widehat L}(C_j)^{-1}-(1-\gamma)J\right)=\frac{1}{\mu(G)}.$$
In that case, $\mu(G)$ is a simple eigenvalue of $L(G)$ and the characteristic
set of $G$ is a singleton consisting of the edge $\{v_i,v_j\}$.
\end{proposition}

Identification of Perron components at a vertex helps to determine
the location of the characteristic set in $G$. The next proposition is a result
in that direction.

\begin{proposition}[\cite{kir1}, Theorem 1 and Corollary 1.1]\label{prop:perron}
Let $G$ be a connected graph. For any vertex $v$ of $G$ which is not a characteristic vertex, the unique Perron
component at $v$ contains the vertices in any characteristic set. A cut-vertex $v$ is a characteristic vertex of $G$ if and only if there are at least
two Perron components of $G$ at $v$, and in that case
$$\mu(G)=\frac{1}{\rho({\widehat L}(C)^{-1})},$$ where $C$ is a
Perron component of $G$ at $v$.
\end{proposition}

\begin{remark}\label{rem:one}
Recall that $P_{2n+1}$ is a path on $2n+1$ vertices with the vertex $v_i$ being
adjacent to the vertex $v_{i+1}$ for $i=1, 2, \ldots, 2n$.
Let $C_1$ and $C_2$ be the two components of $P_{2n+1} - v_{n+1}$.
Then using symmetry, it can be easily observed that the vertex labelled $v_{n+1}$ is the characteristic vertex.
Hence by Proposition~\ref{prop:perron}
$$ \rho({\widehat L}(C)^{-1}) = \dfrac{1}{\mu(P_{2n+1})} = \dfrac{1}{ 2( 1 - \cos(\frac{\pi}{2 n+1}))}.$$
\end{remark}

For non-negative square matrices $A$ and $B$ (not necessarily of the same order), $A\ll B$ means that there exists a permutation matrix $P$ such that $PAP^{T}$ is entri-wise dominated by a principal sub-matrix of $B$, with strict inequality in at least one position in the case $A$ and $B$ have the same order. A useful fact from the Perron-Frobenius theory is that if $B$ is irreducible and $A\ll B$, then $\rho(A)<\rho(B)$. We use this fact together with Propositions \ref{prop:bottleneck} and \ref{prop:alg} (mostly without mention) to get information about the algebraic connectivity of a graph, and in particular, for a tree.

\section{Initial Results}\label{initial}

We start with the following observation. Let $K_n$ be the complete graph of order $n\geq 2$ and $v$ be a vertex of $K_n$. Let $C$ denote the only connected component of $K_n$ at $v$. We have $L(K_n)=nI-J$, an $n$-by-$n$ matrix.
Let $\widehat{L}(C)$ be the principal sub-matrix of $L(K_n)$ corresponding to the vertices of $C$. Then
$\widehat{L}(C)=nI-J$, an $(n-1)$-by-$(n-1)$ matrix. It can be verified that
$\widehat{L}(C)^{-1}=\frac{1}{n}I+\frac{1}{n}J$. Since the sum of each row of $\widehat{L}(C)^{-1}$
is one, the vector of all ones is an eigenvector of $\widehat{L}(C)^{-1}$ corresponding to the eigenvalue one.
So, by the Perron-Frobenius Theorem, the Perron value of $C$ is one.

\begin{lemma}\label{lem:extra}
Let $v$ be a vertex of a connected graph $G$. For a connected component $C$ of $G$ at $v$,
let $\widetilde{C}_v$ denote the induced subgraph of $G$ on the vertices of $C$ together with $v$. Then
the following results hold.
\begin{enumerate}
\item[$(i)$] If $\widetilde{C}_v$ is complete, then the Perron value of $C$ is one.
\item[$(ii)$] If $v$ is a cut-vertex and $\widetilde{C}_v$ is complete for each connected component $C$ of $G$ at $v$, then $\mu(G)=1$.
\end{enumerate}
\end{lemma}

\begin{proof}
$(i)$ The principal sub-matrix of $L(G)$ and that of $L(\widetilde{C}_v)$ corresponding to the vertices of $C$ are the same. If $\widetilde{C}_v$ is complete, then replacing $K_n$ by $\widetilde{C}_v$ in the first paragraph of this section, it follows that the Perron value of $C$ is one.

$(ii)$ By $(i)$, $\rho(\widehat{L}(C)^{-1})=1$ for each connected component $C$ of $G$ at $v$ as $\mu(G)=1$ follows from Proposition \ref{prop:perron}.
\end{proof}

A consequence of Lemma~\ref{lem:extra} is the following.

\begin{corollary}\label{cor:one}
Let $G$ be a connected graph with a cut-vertex $v$. Then $\mu(G)\leq 1.$
\end{corollary}

\begin{proof}
Let $\widehat{G}$ be the graph obtained from $G$ by making $\widetilde{C}_v$ complete for each connected component $C$ of $G$ at $v$. By Lemma \ref{lem:edge}, $\mu(G)\leq \mu(\widehat{G})$. Applying Lemma \ref{lem:extra}$(ii)$ to the pair $(\widehat{G},v)$, we get that $\mu(\widehat{G})=1$. So $\mu(G)\leq 1$.
\end{proof}

The statements of Lemma \ref{lem:extra} and Corollary \ref{cor:one} were implicitly mentioned in (\cite{fal}, Example 1.5, p.51). We write them here with their proofs for the sake of completeness and for later use in this paper. We also note that Corollary \ref{cor:one} follows from (\cite{fi1}, 4.1, p.303) since the vertex connectivity of $G$ is one.

\begin{lemma}\label{lem:extra1}
Let $v$ be a vertex of a connected graph $G$. If $v$ is not a cut-vertex of $G$, then the Perron value of $C=G-v$ (the only connected component of $G$ at $v$) is at least one.
\end{lemma}

\begin{proof}
Let $\widetilde{G}$ be the graph obtained from $G$ by adding one pendant vertex adjacent to $v$. Let $C_1 =C$ and $C_2$ be the two connected components of $\widetilde{G}$ at $v$. The principal sub-matrix of $L(G)$ and that of $L(\widetilde{G})$ corresponding to the vertices of $C_1$ are the same. The Perron value of $C_2$ is $1$ and $\mu(\widetilde{G})\leq 1$ (Corollary \ref{cor:one}). Then the component $C_2$ cannot be a Perron component (use Proposition~\ref{prop:alg}) and
hence $\rho(\widehat{L}(C)^{-1})=\rho(\widehat{L}(C_1)^{-1})\geq \rho(\widehat{L}(C_2)^{-1}) = 1$.
\end{proof}

The next theorem gives conditions under which the equality is attained in Lemma \ref{lem:edge}.

\begin{theorem}\label{thm:equality}
Let $v$ be a vertex of a connected graph $G$ such that $v$ is not a cut-vertex of $G$. Form a new graph $\widetilde{G}$ from $G$ by adding $t\geq 1$ pendant vertices $v_1,\cdots,v_t$ adjacent to $v$. Let $\widehat{G}$ be the graph obtained from $\widetilde{G}$ by adding $k,\; 0\leq k\leq \frac{t(t-1)}{2},$ edges among $v_1,\cdots,v_t$. Then $\mu(\widetilde{G})=\mu(\widehat{G}).$
\end{theorem}

\begin{proof}
Let $G^*$ be the graph obtained from $\widetilde{G}$ by adding $\frac{t(t-1)}{2}$ edges among $v_1,\cdots,v_t$. By Lemma \ref{lem:edge} and Corollary \ref{cor:one}, $\mu(\widetilde{G})\leq \mu(\widehat{G})\leq \mu(G^*)\leq 1$. If $\mu(\widetilde{G})=1$, we are done. So, we assume that $\mu(\widetilde{G})<1$
and prove that  $\mu(G^*)\leq\mu(\widetilde{G})$.

Let $C_1,C_2,\cdots,C_{t+1}$ be the $t+1$ connected components of $\widetilde{G}$ at $v$, where $C_1=G-v$. For $2\leq i\leq t+1$, the Perron value of $C_i$ is one. As $v$ is not a cut-vertex of $G$, by Lemma~\ref{lem:extra1},
the Perron value of $C_1$ is at least $1$. So, by definition, $C_1$ is a Perron component of $\widetilde{G}$ at $v$. By Proposition \ref{prop:alg}$(i)$, there is a unique non-negative number $\theta$ such that
$$\lambda\left({\widehat L}(C_1)^{-1}-\theta J\right)=\rho\left({\widehat
L}(C_2)^{-1}\oplus\cdots\oplus{\widehat L}(C_{t+1})^{-1}\oplus
[0]+\theta J\right)=\frac{1}{\mu(\widetilde{G})}.$$
Since $\mu(\widetilde{G})<1$, the second equality implies that $\theta>0$.

Let $M={\widehat
L}(C_2)^{-1}\oplus\cdots\oplus{\widehat L}(C_{t+1})^{-1}\oplus
[0]+\theta J$. Thus $M$ is a $(t+1)$-by-$(t+1)$ positive matrix with Perron value $\frac{1}{\mu(\widetilde{G})}$. Let $Y=[y_1,\cdots,y_t,y_{t+1}]^T$ be a Perron vector of $M$ and take $\beta=\theta(y_1+\cdots +y_{t+1})$. Then $\beta>0$. A simple calculation shows that $y_i +\beta =\frac{1}{\mu(\widetilde{G})}y_i$ for $1\leq i\leq t$ and $\beta =\dfrac{1}{\mu(\widetilde{G})}y_{t+1}$. Since $0<\mu(\widetilde{G})<1,$ it follows that $y_1=\cdots =y_t$. So without loss of generality we may take $Y=[1,\cdots,1,y]^T$ for some $y>0$.

There are two connected components of $G^*$ at $v$, say $D_1=C_1$ and $D_2$. The Perron value of $D_2$ is one by Lemma \ref{lem:extra}$(i)$. So $D_1$ is a Perron component of $G^*$ at $v$ since it has Perron value at least one. Let $M^*= \widehat{L}(D_2)^{-1}\oplus [0]+\theta J$.
Using $\widehat{L}(D_2)^{-1}=\frac{1}{t+1}I+\frac{1}{t+1}J$ (see the first paragraph of this section), it can be verified that $M^*Y=\frac{1}{\mu(\widetilde{G})}Y.$ So, by the Perron-Frobenius Theorem, $\rho(M^*)=\frac{1}{\mu(\widetilde{G})}$. Thus,
$$\rho\left(\widehat{L}(D_2)^{-1}\oplus [0]+\theta J\right)=\frac{1}{\mu(\widetilde{G})}=\lambda\left({\widehat L}(C_1)^{-1}-\theta J\right)=\lambda\left({\widehat L}(D_1)^{-1}-\theta J\right).$$
Now, by Proposition \ref{prop:alg}$(ii)$, $\mu(\widetilde{G})$ is an eigenvalue of $L(G^*)$ and hence $\mu(G^*)\leq\mu(\widetilde{G})$ as, by definition, $\mu(G^*)$ is the smallest nonzero eigenvalue of $L(G^*)$.
This completes the proof.
\end{proof}

\section{With Pendant Vertices}\label{with}

Let $\mathfrak{H}_{n,k}$ denote the class of all connected graphs of order $n$ with $k$ pendant vertices. We may assume that $1\leq k\leq n-1$ and hence $n\geq 3$.  We first consider the question of maximizing the algebraic connectivity over $\mathfrak{H}_{n,k}$. We now define a collection of graphs, denoted $P_n^k$, which have exactly $k$ pendant vertices and
hence are members of $\mathfrak{H}_{n,k}$.
For $k\neq n-2$, the graph $P_n ^k$ is obtained by adding $k$ pendant vertices adjacent to a single vertex of the complete graph $K_{n-k}$. There is no graph of order $n=3$ with exactly one pendant vertex. For $n\geq 4$ and $k=n-2$, the graph
$P_n ^{n-2}$ is obtained by adding $n-3$ pendant vertices adjacent to a pendant vertex of the path $P_3$.

\begin{figure}[h]
\centering
\includegraphics[scale=.9]{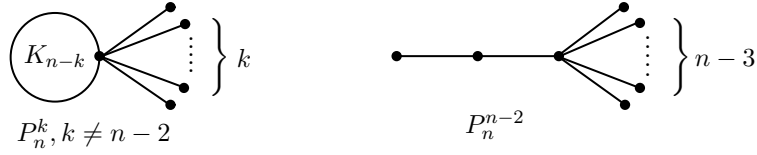}
\caption{The graph $P_n^k$}\label{The graph PNK}
\end{figure}

\begin{theorem}\label{thm:max}
The graph $P_n^k,\; k\neq n-2,$ attains the maximum algebraic connectivity over $\mathfrak{H}_{n,k}$ and this maximum value is equal to one.
\end{theorem}

\begin{proof}
Let $G$ be a graph in $\mathfrak{H}_{n,k}$. Since $k\geq 1$, $G$ has a cut-vertex. By Corollary \ref{cor:one}, $\mu(G)\leq 1$. Also $P_n^k,\;k\neq n-2,$ has a cut-vertex satisfying the hypothesis of Lemma \ref{lem:extra}$(ii)$. So $\mu(P_n^k)=1$ and hence $\mu(G)\leq \mu(P_n^k)$.
\end{proof}

\begin{remark}[{\bf Uniqueness}] \label{rem:uniqueness}
For $k=n-1$, $\mathfrak{H}_{n,k}$ contains only one graph, namely, $K_{1,n-1}\simeq P_n^{n-1}$. For $k=n-3$, let $G$ be a graph in $\mathfrak{H}_{n,n-3}$. If $G$ is a tree then clearly $\mu(G)<1$. Otherwise $G$ has a unique 3-cycle. By (\cite{fal}, Theorem 4.14, p.73), it follows that $\mu(G)=1$ if and only if $G\simeq P_n^{n-3}$.
However, for $1\leq k\leq n-4$, $P_n^k$ is not the unique graph (up to isomorphism) in $\mathfrak{H}_{n,k}$ having algebraic connectivity one. We give an example below.
\end{remark}

\begin{example}
Let $1\leq k\leq n-4$. Let $v_{1},\cdots,v_{n-k-1}$ be $n-k-1$ pendant vertices in the star $K_{1,n-1}$. Let $G$ be the graph obtained from $K_{1,n-1}$ by adding new edges $\{v_i,v_{i+1}\}$, $1\leq i\leq n-k-2$. Then $G$ is a member of $\mathfrak{H}_{n,k}$. Since $n\geq 5$, $G$ is not isomorphic to $P_n^k$. But by Theorem $\ref{thm:equality}$,  $\mu(G)=\mu(K_{1,n-1})=1$.
\end{example}

\begin{theorem}
The graph $P_n^{n-2}$ uniquely maximizes the algebraic connectivity over $\mathfrak{H}_{n,n-2}$.
\end{theorem}

\begin{proof}
Any graph in $\mathfrak{H}_{n,n-2}$ is a tree with diameter three. By Proposition \ref{prop:diameter1}, the maximum algebraic connectivity is uniquely attained by the tree $T^{4}_{n-4}\simeq P_n^{n-2}$ over $\mathfrak{H}_{n,n-2}$.
\end{proof}

We next consider the question of minimizing the algebraic connectivity over $\mathfrak{H}_{n,k}$. We first prove
the result for  $k=1$ (and so $n\geq 4$). We denote by $C_3^{n-3}$ the graph obtained by adding a path of order $n-3$ to a vertex of $K_3$ (see Figure~\ref{The graph C3n-3}). Then $C_3^{n-3}$ is a member of $\mathfrak{H}_{n,1}$.

\begin{figure}[h]
\centering
\includegraphics[scale=.9]{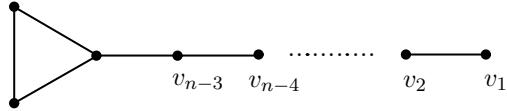}
\caption{The graph $C_3^{n-3}$}\label{The graph C3n-3}
\end{figure}

\begin{theorem}\label{thm:min1}
The graph $C_3^{n-3}$ uniquely attains the minimum algebraic connectivity over $\mathfrak{H}_{n,1}$.
\end{theorem}

\begin{proof}
For $n=4,$ $C_3^1$ is the only graph in $\mathfrak{H}_{4,1}$. For $n=5,$ $C_4^1$ (adding a pendant vertex adjacent to a vertex of a $4$-cycle) is the only graph in $\mathfrak{H}_{5,1}$, non-isomorphic to $C_3^2$, with least possible edges. The other graphs, non-isomorphic to $C_3^2$, in $\mathfrak{H}_{5,1}$ can be obtained from $C_4^1$ by joining some non-adjacent vertices. So, for any other graph $G$ on $5$ vertices and with one pendant vertex,
$$\mu(G) \ge \mu(C_4^1) \approx 0.8299  > 0.5188 \approx \mu(C_3^2).$$

Let $n\geq 6$. Any graph in $\mathfrak{H}_{n,1}$ has at least one cycle. Let $G$ be a graph in $\mathfrak{H}_{n,1}$ which is not isomorphic to $C_3^{n-3}$. Delete some of the edges from the cycles of $G$ to get a tree $G_1$. Further, we delete these edges in such a way that $G_1$ is neither isomorphic to the path $P_n$ nor to the graph $G_2$
($G$ is not isomorphic to $C_3^{n-3}$) as shown in Figure \ref{The tree G_2}.
\begin{figure}[h]
\centering
\includegraphics[scale=.9]{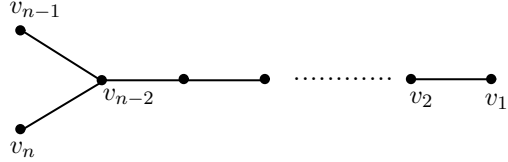}
\caption{The tree $G_2$}\label{The tree G_2}
\end{figure}
This is possible since $n\geq 6$. By Lemma \ref{lem:edge}, $\mu(G_1)\leq \mu(G)$. Starting with $G_1$, by a finite sequence of graph operations
consisting of grafting of edges, we can get the tree $G_2$ from $G_1$.
Also note that we must get the graph $\widehat{G}_2$ as shown in Figure \ref{The tree G2hat} in the penultimate step of getting $G_2$.
\begin{figure}[h]
\centering
\includegraphics[scale=.9]{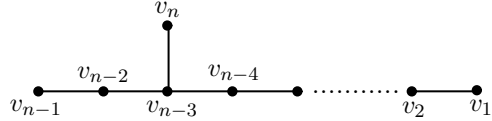}
\caption{The tree $\widehat{G}_2$}\label{The tree G2hat}
\end{figure}
Since $G_2\simeq T(2,1,n-3)$, $\mu(G_2)<\mu(\widehat{G}_2)$ by Proposition \ref{prop:diameter} and  $\mu(\widehat{G}_2)\leq\mu(G_1)$ by Proposition \ref{prop:gkl}. Now, the graph $C_3^{n-3}$ is obtained from $G_2$ by adding the edge $\{v_{n-1}, v_n\}$. By Theorem
\ref{thm:equality}, $\mu(C_3^{n-3})=\mu(G_2)$ and it follows that $\mu\left(C_3^{n-3}\right)<\mu(G)$. This completes the proof.
\end{proof}

We now consider the case $k\geq 2$. Recall that the tree $T\left(\lceil \frac{k}{2} \rceil,\lfloor \frac{k}{2}\rfloor, n-k\right)$ of order $n$ has $k$ pendant vertices and diameter $n-k+1$.

\begin{theorem}\label{thm:min}
The tree $T\left(\lceil \frac{k}{2} \rceil,\lfloor \frac{k}{2}\rfloor, n-k\right)$ uniquely attains the minimum algebraic connectivity over $\mathfrak{H}_{n,k}$, $k\geq 2$.
\end{theorem}

\begin{proof}
The result is obvious for $k=2$. Let $k\geq 3$ and $G$ be a graph in $\mathfrak{H}_{n,k}$ which is not isomorphic to $T\left(\lceil \frac{k}{2} \rceil,\lfloor \frac{k}{2}\rfloor, n-k\right)$. If $G$ is a tree then the result follows
from Proposition~\ref{prop:diameter}. So, let us assume that $G$ is not a tree.
Now, carefully remove some of the edges from the cycles of $G$, to get a tree $\widehat{G}$
which is non-isomorphic to $T\left(\lceil \frac{k}{2} \rceil,\lfloor \frac{k}{2}\rfloor, n-k\right).$ Since $\widehat{G}$ has at least $k$ pendant vertices, it has diameter at most $n-k+1.$ If the diameter of $\widehat{G}$ is strictly less than $n-k+1,$ then form a new tree $\widetilde{G}$ of diameter $n-k+1$ from $\widehat{G}$ by grafting of edges such that $\widetilde{G}$ is not isomorphic to $T\left(\lceil \frac{k}{2} \rceil,\lfloor \frac{k}{2}\rfloor, n-k\right)$. Now, Lemma \ref{lem:edge} and Propositions \ref{prop:diameter} and \ref{prop:gkl} complete the proof.
\end{proof}

Let $\mathfrak{T}_{n,k}$ denote the subclass of $\mathfrak{H}_{n,k}$ consisting of all trees of order $n$ with $k$, $2\leq k\leq n-1$, pendant vertices. By Theorem \ref{thm:min}, the tree $T(\lceil \frac{k}{2} \rceil,\lfloor \frac{k}{2}\rfloor, n-k)$ has the minimum algebraic connectivity over $\mathfrak{T}_{n,k}$. Now the following natural question arises: Which tree attains the maximum algebraic connectivity over $\mathfrak{T}_{n,k}$? We answer this question in Theorem \ref{thm:treemax} below.

Let $q=\lfloor \frac{n-1}{k} \rfloor$ and $n-1=kq+r,\; 0\leq r\leq k-1$. Let $T_{n,k}$ be the tree obtained by adding $r$ paths each of length $q$ and $k-r$ paths each of length $q-1$ to a given vertex. Then $T_{n,k}$ is a member of $\mathfrak{T}_{n,k}$.

\begin{figure}[h]
\centering
\includegraphics[scale=.9]{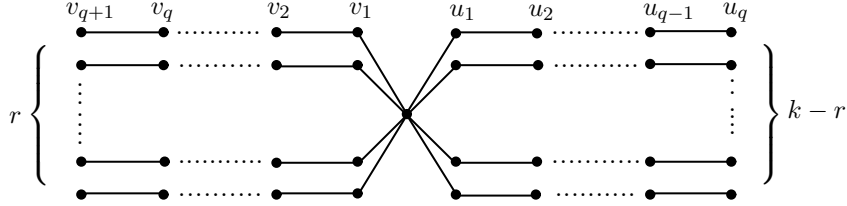}
\caption{The tree $T_{n,k}$}\label{starpath tree}
\end{figure}

\begin{theorem}\label{thm:treemax}
The tree $T_{n,k}$ uniquely attains the maximum algebraic connectivity over $\mathfrak{T}_{n,k}$.
\end{theorem}

\begin{proof}
For $k=2$, $T_{n,2}\simeq P_n$ is the only element in $\mathfrak{T}_{n,2}$ and the result is obvious. Assume that $k\geq 3$. Let $v$ be the vertex of $T_{n,k}$ of degree $k$. There are $k$ connected components of $T_{n,k}$ at $v$, and $r$ of them are paths of order $q+1$ and the rest $k-r$ components are paths of order $q$. Let $d_k$ be the diameter of $T_{n,k}$. Then
\begin{equation*}
d_k=\left\{\begin{array}{ll}
  2q & \text{ if }r=0 \\
  2q+1 & \text{ if }r=1\\
  2q+2 & \text{ if }2\leq r\leq k-1
\end{array}.\right.
\end{equation*}
Any tree in $\mathfrak{T}_{n,k}$ has diameter at least $d_k$, and equality holds if and only if the tree is isomorphic to $T_{n,k}$.

If $r=0$ then, by Proposition \ref{prop:perron}, each connected component of $T_{n,k}$ at $v$ is a Perron component
consisting of paths of order $q$ and
therefore by Remark~\ref{rem:one}, $\mu(T_{n,k})=2\left(1-\cos\frac{\pi}{2q +1}\right)$. If $r\geq 1$, then from Proposition \ref{prop:bottleneck} together with the Perron-Frobenius theory, it follows that the Perron components of $T_{n,k}$ at $v$ are the ones with $q+1$ vertices. So $\mu(T_{n,k})\geq 2\left(1-\cos\frac{\pi}{2q+3}\right)$ with strict inequality if $r=1$. Note that if $T$ is any tree of order $n$ with diameter $d$, then $\mu(T)\leq 2\left(1-\cos \frac{\pi}{d+1} \right)$ (\cite{gr2}, Corollary 4.4, p.234). Thus, for any tree $T$ in $\mathfrak{T}_{n,k}$ with diameter $\beta\geq d_k +1$, we have
$$\mu(T)\leq 2\left(1-\cos \frac{\pi}{\beta +1} \right)\leq 2\left(1-\cos\frac{\pi}{d_k +1}\right)\leq\mu(T_{n,k}).$$
Since the second inequality is strict if $r=0$ and $r\geq 2$, and the last inequality is strict if $r=1$, the uniqueness of $T_{n,k}$ having the maximum algebraic connectivity over $\mathfrak{T}_{n,k}$ follows. This completes the proof.
\end{proof}

\section{Without Pendant Vertex}\label{without}

Let $\mathcal{F}_n,n\geq 3,$ be the class of all connected graphs of order $n$ without any pendant vertex. The complete graph $K_n$ is a member of $\mathcal{F}_n$ and it has the maximum algebraic connectivity, namely $\mu(K_n)=n$. Here, our aim is to find the graph that has the minimum algebraic connectivity over $\mathcal{F}_n$. By Lemma \ref{lem:edge}, we consider only those graphs in $\mathcal{F}_n$ with minimum possible edges (if we delete any edge from such a graph, then the new graph is either disconnected or has atleast one pendant vertex). Note that any two cycles in such a graph are edge disjoint. For $n=3,4$, the cycle $C_n$ of order $n$ is the only such graph. In \cite{fi1}, Fiedler showed that $\mu(C_m)=2\left(1-\cos\frac{2\pi}{m}\right)$. So $\mu(C_3)=3$ and $\mu(C_4)=2$. For $n=5,$ there are two such graphs $G_1$ and $G_2$ in $\mathcal{F}_5$, up to isomorphism, where $G_1$ is $C_5$ and $G_2$ is the graph with two $3$-cycles having one common vertex. By Lemma \ref{lem:extra}$(ii)$, $\mu(G_2)=1$ and hence,
$$\mu(G_1) = 2\left(1-\cos\frac{2\pi}{5}\right) > 1 = \mu(G_2).$$

We next consider when $n\geq 6$. Let $C_{\lfloor \frac{n}{2}\rfloor,\lceil \frac{n}{2}\rceil}$ be the graph with two cycles $C_{\lfloor \frac{n+1}{2}\rfloor}$ and $C_{\lceil \frac{n+1}{2}\rceil}$ with exactly one common vertex. Then $C_{\lfloor \frac{n}{2}\rfloor,\lceil \frac{n}{2}\rceil}$ is a member of $\mathcal{F}_n$. We first prove the following.

\begin{figure}[h]
\centering
\includegraphics[scale=.9]{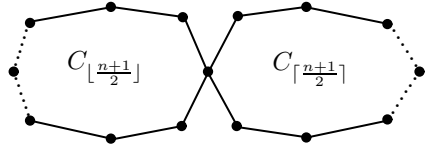}
\caption{The graph $C_{\lfloor \frac{n}{2}\rfloor,\lceil \frac{n}{2}\rceil}$}\label{twocircles:1}
\end{figure}

\begin{lemma}\label{lem:cn}
For $n\geq 6,$ $\mu\left(C_{\lfloor \frac{n}{2}\rfloor,\lceil \frac{n}{2}\rceil}\right)<\mu(C_n).$
\end{lemma}

\begin{proof}
Let $v$ be the cut-vertex of $C_{\lfloor \frac{n}{2}\rfloor,\lceil \frac{n}{2}\rceil}$. Let $C_1$ and $C_2$ be the two connected components of $C_{\lfloor \frac{n}{2}\rfloor,\lceil \frac{n}{2}\rceil}$ at $v$ of orders $\lfloor \frac{n-1}{2}\rfloor$ and $\lceil \frac{n-1}{2} \rceil$, respectively. For $i=1,2$, observe that
{$$\widehat{L}(C_i)= \left[
\begin{matrix}
 2 & -1 & 0 & \cdots & 0 & 0\\
 -1 & 2 & -1 & 0 & \cdots & 0\\
 0 & -1 & 2 & -1 & \cdots\\
 \vdots& \ddots & \ddots & \ddots & \ddots & \vdots\\
 0 & \cdots & 0 & -1 & 2 & -1\\
 0 & \cdots & 0 & 0 & -1 & 2
\end{matrix}
\right].$$}
It is well known that the smallest eigenvalue of $\widehat{L}(C_1)$ is $2\left(1-\cos\frac{\pi}{\lfloor \frac{n+1}{2}\rfloor}\right)$ and that of $\widehat{L}(C_2)$ is  $2\left(1-\cos\frac{\pi}{\lceil \frac{n+1}{2}\rceil}\right)$. If $n$ is odd, then $\rho(\widehat{L}(C_1)^{-1})=\rho(\widehat{L}(C_2)^{-1})$ and hence $C_1$ and $C_2$ both are Perron components of $C_{\lfloor \frac{n}{2}\rfloor,\lceil \frac{n}{2}\rceil}$ at $v$. So by Proposition \ref{prop:perron}, $$\mu\left(C_{\lfloor \frac{n}{2}\rfloor,\lceil \frac{n}{2}\rceil}\right)=
2\left(1-\cos\frac{2\pi}{n+1}\right)<2\left(1-\cos\frac{2\pi}{n}\right)=\mu(C_n).$$
If $n$ is even, then $C_2$ is the only Perron component of $C_{\lfloor \frac{n}{2}\rfloor,\lceil \frac{n}{2}\rceil}$ at $v$ and it follows that
$$2\left(1-\cos\frac{2\pi}{n+2}\right)<\mu\left(C_{\lfloor \frac{n}{2}\rfloor,\lceil \frac{n}{2}\rceil}\right)<2\left(1-\cos\frac{2\pi}{n}\right)=\mu(C_n).$$
This completes the proof.
\end{proof}

\begin{figure}[h]
\centering
\includegraphics[scale=.9]{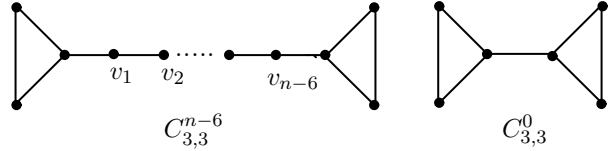}
\caption{The graph $C_{3,3}^{n-6}$}\label{twocircles:2}
\end{figure}

Now, take a path of order $n-6$ and add a $3$-cycle to each of its two pendant vertices. The new graph thus obtained is a member of $\mathcal{F}_n$. We denote it by $C_{3,3}^{n-6}$.

\begin{theorem}\label{lem:min0}
For $n\geq 6$, the graph $C_{3,3}^{n-6}$ uniquely attains the minimum algebraic connectivity over $\mathcal{F}_n$.
\end{theorem}

\begin{proof}
In $\mathcal{F}_n$, $C_n$ is the only graph whose maximum vertex degree is two. Since we are minimizing the algebraic connectivity over $\mathcal{F}_n$, by Lemma \ref{lem:cn} we may consider graphs with maximum vertex degree at least three. Let $G_1$ be such a graph in $\mathcal{F}_n$ not isomorphic to $C_{3,3}^{n-6}$. Then $G_1$ has at least two cycles. Form a tree $G_2$ by deleting one edge from each cycle of $G_1$ in such a way that $G_2$ has diameter at most $n-3$ and if equality holds, then $G_2$ is not isomorphic to $T(2,2,n-4)$ (see Figure \ref{The tree T(2,2,n-4)}). By Lemma \ref{lem:edge}, $\mu(G_2)\leq \mu(G_1).$
\begin{figure}[h]
\centering
\includegraphics[scale=.9]{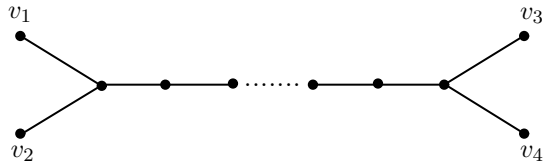}
\caption{The tree $T(2,2,n-4)$}\label{The tree T(2,2,n-4)}
\end{figure}
If the diameter of $G_2$ is less than $n-3,$ then form a new tree $G_3$ from $G_2$, by grafting of edges,
such that the diameter of $G_3$ is $n-3$ and $G_3$ is not isomorphic to $T(2,2,n-4)$. By Proposition \ref{prop:gkl}, $\mu(G_3)\leq \mu(G_2)$. By Proposition \ref{prop:diameter}, $\mu(T(2,2,n-4))< \mu(G_3)$. Now, $C_{3,3}^{n-6}$ is isomorphic to the graph obtained from $T(2,2,n-4)$ by adding two new edges $\{v_1,v_2\}$ and $\{v_3,v_4\}$. Since  $\mu(T(2,2,n-4))=\mu\left(C_{3,3}^{n-6}\right)$ by Theorem \ref{thm:equality}, it follows that $\mu\left(C_{3,3}^{n-6}\right)< \mu(G_1)$. This completes the proof.
\end{proof}

\section{Unicyclic Graphs}\label{unicyclic}

The algebraic connectivity of unicyclic graphs of order $n$ with fixed girth has been extensively studied by Fallat et al. \cite{fal,fal2,fal3}. Let $\mathcal{U}_n, \;n\geq 3,$ denote the class of all unicyclic graphs of order $n$. Here we discuss the question of extremizing the algebraic connectivity over $\mathcal{U}_n$. For $n=3$, $\mathcal{U}_n$ contains only one graph. So we assume that $n\geq 4$.

First consider the maximizing case.
For the cycle $C_n$ of order $n$, we have $\mu(C_n)> 1$ if and only if $n\leq 5$, and $\mu(C_n)= 1$ if and only if $n=6$.
Also, a graph in $\mathcal{U}_n$, which is not isomorphic to $C_n$, has at least one cut-vertex and hence it has algebraic connectivity at most one (Corollary \ref{cor:one}). Again note that the graph $P_n^{n-3}$ (see Figure \ref{The graph PNK}) is a member of $\mathcal{U}_n$ and $\mu(P_n^{n-3})=1$ (Lemma \ref{lem:extra}$(ii)$).
Therefore, if $n \le 6$, we have the following: \\
 for $n\leq 5$, clearly $C_n$ is the only graph maximizing the algebraic connectivity over $\mathcal{U}_n$ and for $n=6$, $C_6$ and $P_6^{3}$ are two non-isomorphic graphs in $\mathcal{U}_n$ having algebraic connectivity one.

Now, suppose that $n\geq 6$ and let $G$ be a graph in $\mathcal{U}_n$ with at least one cut-vertex which is not isomorphic to $P_n^{n-3}$. Let $g$ be the girth of $G$. If $g=3$, then $\mu(G)<1$ by (\cite{fal}, Theorem 4.14, p.73).
If $g=4,5$ or $6$ and $G_1$ is the graph in $\mathcal{U}_{g+1}$ obtained by adding one pendant vertex adjacent to a vertex of $C_g$. Then $\mu(G_1)\approx 0.8299$ if $g=4$, $\mu(G_1)\approx 0.6972$ if $g=5$ and $\mu(G_1)\approx 0.5858$ if $g=6$. Since $G$ is obtained from $G_1$ by a finite sequence of addition of pendant vertices, by Lemma~\ref{lem:addvertex}, $\mu(G)\leq\mu(G_1)<1$.  For $g\geq 7$, $G$ can be obtained from $C_g$ by a finite sequence of graph operations consisting of addition of pendant vertices. So, by Lemma~\ref{lem:addvertex}, $\mu(G)\leq\mu(C_g)<1$.  Thus, we have the following result.

\begin{theorem}
The maximum algebraic connectivity over $\mathcal{U}_n$ is uniquely attained by $C_n$ if $n\leq 5$ and uniquely attained by $P_n^{n-3}$ if $n> 6$. When $n=6$, $C_6$ and $P_6^{3}$ are the only two graphs, up to isomorphism, having the maximum algebraic connectivity over $\mathcal{U}_6$.
\end{theorem}

We next consider the minimizing case and prove the following.
\begin{theorem}
The graph $C_3^{n-3}$ (see Figure~\ref{The graph C3n-3}) uniquely attains the minimum algebraic connectivity over $\mathcal{U}_n$.
\end{theorem}
\begin{proof}
If $n=4$, then $C_4$ and $C_3^{1}$ are the only two graphs, up to isomorphism, in $\mathcal{U}_4$ with $\mu(C_3^{1})=1<\mu(C_4)$. If $n=5$, then $\mathcal{U}_5$ contains precisely five graphs (up to isomorphism). They are $C_3^2,C_5,P_5^2$, $C_4^1$ (adding a pendant vertex to a vertex of $C_4$) and $H$ (adding two pendant vertices to two distinct vertices of $C_3$). We have $\mu(C_3^2)\approx 0.5188$, $\mu(H)\approx 0.6972$, $\mu(C_4^1)\approx 0.8299$, $\mu(P_5^2)=1$ and $\mu(C_5)>1$. Thus $C_3^{2}$ has the minimum algebraic connectivity.

Now, let $n\geq 6$. Let $G$ be a graph in $\mathcal{U}_n$ with at least one pendant vertex which is not isomorphic to $C_3^{n-3}$. By the same argument used in the proof of Theorem \ref{thm:min1}, we can form a graph $G_2$ (see Figure \ref{The tree G_2}) from $G$ such that $\mu(C_3^{n-3})=\mu(G_2)< \mu(G)$. Thus, we only need to show that $\mu(G_2)< \mu(C_n)$. In $G_2$,  it can be verified that $\left\{v_{\lfloor\frac{n}{2}\rfloor},v_{\lfloor\frac{n}{2}\rfloor+1}\right\}$ is the characteristic edge~(\cite{patra1}, Lemma 2.2, p.386). So, in $G_2 - v_{\lfloor\frac{n}{2}\rfloor}$, the component $C$, consisting of the path on $\lfloor\frac{n}{2}\rfloor -1$ vertices is not a Perron component. Therefore,
$$\mu(G_2) < \rho(\widehat(C)^{-1}) = 2\left(1-\cos\frac{\pi}{2(\lfloor\frac{n}{2}\rfloor-1)+1} \right) < 2\left(1-\cos\frac{2\pi}{n}\right)=\mu(C_n).$$
Thus, we have the required result.
\end{proof}

\end{document}